\documentclass[letterpaper,10pt]{amsart}

\usepackage[all]{xy}                        %

\CompileMatrices                            

\UseTips                                    

\input xypic
\usepackage[bookmarks=true]{hyperref}       

\usepackage{amssymb,latexsym,amsmath,amscd}
\usepackage{xspace}
\usepackage{color}
\usepackage{graphicx}
\usepackage{stmaryrd}

\usepackage{tikz}

\reversemarginpar

\theoremstyle{plain}
\newtheorem{theorem}{Theorem}[section]
\newtheorem*{theorem*}{Theorem}
\newtheorem{proposition}[theorem]{Proposition}
\newtheorem{corollary}[theorem]{Corollary}
\newtheorem{lemma}[theorem]{Lemma}

\theoremstyle{definition}
\newtheorem{definition}[theorem]{Definition}

\newtheorem{remark}[theorem]{Remark}
\newtheorem{example}[theorem]{Example}

\newcommand{\enm}[1]{\ensuremath{#1}}          %
\newcommand{\op}[1]{\operatorname{#1}}
\newcommand{\cal}[1]{\mathcal{#1}}

\newcommand{\BB}{\enm{\mathbb{B}}}
\newcommand{\CC}{\enm{\mathbb{C}}}

\newcommand{\RR}{\enm{\mathbb{R}}}

\newcommand{\ZZ}{\enm{\mathbb{Z}}}

\newcommand{\PP}{\enm{\mathbb{P}}}

\newcommand{\Aa}{\enm{\cal{A}}}
\newcommand{\Bb}{\enm{\cal{B}}}

\newcommand{\Dd}{\enm{\cal{D}}}
\newcommand{\Ee}{\enm{\cal{E}}}
\newcommand{\Ff}{\enm{\cal{F}}}
\newcommand{\Gg}{\enm{\cal{G}}}
\newcommand{\Hh}{\enm{\cal{H}}}
\newcommand{\Ii}{\enm{\cal{I}}}

\newcommand{\Ll}{\enm{\cal{L}}}
\newcommand{\Mm}{\enm{\cal{M}}}

\newcommand{\Oo}{\enm{\cal{O}}}

\newcommand{\Rr}{\enm{\cal{R}}}
\newcommand{\Ss}{\enm{\cal{S}}}

\renewcommand{\phi}{\varphi}
\renewcommand{\theta}{\vartheta}
\renewcommand{\epsilon}{\varepsilon}

\newcommand{\Pic}{\op{Pic}}

\newcommand{\Ext}{\op{Ext}}

\newcommand{\rk}{\op{rank}}

\renewcommand{\to}[1][]{\xrightarrow{\ #1\ }}

\renewcommand{\a}{\alpha}

\newcommand{\old}[1]{}

\begin{document}

\title[Logarithmic co-Higgs bundles]{Logarithmic co-Higgs bundles}
\author{Edoardo Ballico and Sukmoon Huh}
\address{Universit\`a di Trento, 38123 Povo (TN), Italy}
\email{edoardo.ballico@unitn.it}
\address{Sungkyunkwan University, 300 Cheoncheon-dong, Suwon 440-746, Korea}
\email{sukmoonh@skku.edu}
\keywords{co-Higgs bundle, double covering, nilpotent}
\thanks{The first author is partially supported by MIUR and GNSAGA of INDAM (Italy). The second author is supported by Basic Science Research Program 2015-037157 through NRF funded by MEST and the National Research Foundation of Korea(KRF) 2016R1A5A1008055 grant funded by the Korea government(MSIP)}

\subjclass[2010]{Primary: {14J60}; Secondary: {14D20, 53D18}}

\begin{abstract}
In this article we introduce a notion of logarithmic co-Higgs sheaves associated to a simple normal crossing divisor on a projective manifold, and show their existence with nilpotent co-Higgs fields for fixed ranks and second Chern classes. Then we deal with various moduli problems with logarithmic co-Higgs sheaves involved, such as coherent systems and holomorphic triples, specially over algebraic curves of low genus.
\end{abstract}

\maketitle

\section{Introduction}

A co-Higgs sheaf on a complex manifold $X$ is a torsion-free coherent sheaf $\Ee$ on $X$ together with an endomorphism $\Phi$ of $\Ee$, called a {\it co-Higgs field}, taking values in the tangent bundle $T_X$ of $X$, i.e. $\Phi \in H^0(\mathcal{E}nd(\Ee)\otimes T_X)$, such that the integrability condition $\Phi \wedge \Phi=0$ is satisfied. When $\Ee$ is locally free, it is a generalized vector bundle on $X$, considered as a generalized complex manifold and it is introduced and developed by Hitchin and Gualtieri in \cite{Hi, Gual}. A naturally defined stability condition on co-Higgs sheaves allows one to study their moduli spaces and Rayan and Colmenares investigate their geometry over projective spaces and a smooth quadric surface in \cite{R2, Rayan} and \cite{VC1}. Indeed it is expected that the existence of stable co-Higgs bundles forces the position of $X$ to be located at the lower end of the Kodaira spectrum, and Corr\^{e}a shows in \cite{Correa} that a K\"ahler compact surface with a nilpotent stable co-Higgs bundle of rank two is uniruled up to finite \'etale cover. In \cite{BH1, BH} the authors suggest a simple way of constructing nilpotent co-Higgs sheaves, based on Hartshorne-Serre correspondence, and obtain some (non-)existence results.

In this article we investigate the existence of nilpotent co-Higgs sheaves with a co-Higgs field vanishing in the normal direction to a given divisor of $X$; for a given arrangement $\Dd$ of smooth irreducible divisors of $X$ with simple normal crossings, the sheaf $T_X(-\log \Dd)$ of logarithmic vector fields along $\Dd$ is locally free and we consider a pair $(\Ee, \Phi)$ of a torsion-free coherent sheaf $\Ee$ and a morphism $\Phi : \Ee \rightarrow \Ee \otimes T_X(-\log \Dd)$ with the integrability condition satisfied. The pair is called a {\it $\Dd$-logarithmic co-Higgs sheaf} and it is called $2$-nilpotent if $\Phi \circ \Phi$ is trivial. Our first result is on the existence of nilpotent $\Dd$-logarithmic co-Higgs sheaves of rank at least two. 

\begin{theorem}[Propositions \ref{aa1}, \ref{aa2} and \ref{aa2.00}]\label{thm33}
Let $X$ be a projective manifold with $\dim (X)\ge 2$ and $\Dd\subset X$ be a simple normal crossing divisor. For fixed $\Ll\in \Pic (X)$ and an integer $r\ge 2$, there exists a $2$-nilpotent $\Dd$-logarithmic co-Higgs sheaf $(\Ee, \Phi)$, where $\Phi\ne 0$ and $\Ee$ is reflexive and indecompodable with $c_1(\Ee)\cong \Ll$ and $\rk \Ee = r$. 
\end{theorem}

Indeed, we can strengthen the statement of Theorem \ref{thm33} by requiring $\Ee$ to be locally free, in cases $\dim (X)=2$ or $r\ge \dim (X)$, due to the statement of the Hartshorne-Serre correspondence and the dimension of non-locally free locus (see Propositions \ref{aa1} and \ref{aa2}). Moreover, in case $\dim (X)=2$, we suggest an explicit number such that a logarithmic co-Higgs bundle exists for each second Chern class at least that number. We notice that the logarithmic co-Higgs sheaves constructed in Theorem \ref{thm33} are hightly unstable, which is consistent with the general philosophy on the existennce of stable co-Higgs bundles (see \cite[Theorem 1.1]{Correa} for example). 

Then we pay our attention to various different types of semistable objects involving logarithmic co-Higgs sheaves. In Section \ref{exp} we produce several examples of nilpotent semistable logarithmic co-Higgs sheaves on projective spaces and a smooth quadric surface, using a simple way of constructin in \cite{BH}. Since the logarithmic co-Higgs sheaves are co-Higgs sheaves in the usual sense with an additional vanishing condition in the normal direction of divisors, so their moduli space is a closed subvariety of the moduli of the usual co-Higgs sheaves. In Section \ref{mex} we describe two moduli spaces of logarithmic co-Higgs bundles of rank two on $\PP^2$ in two cases. 

Then in Section \ref{osem} we experiment with extensions of the notion of stability for co-Higgs sheaves and logarithmic co-Higgs sheaves. A key point for the study of moduli spaces was the introduction of parameters for the conditions of stability. We extend two of them, coherent systems and holomorphic triples, to co-Higgs sheaves. Specially in case of holomorphic triples, we show that any holomorphic triple admits the Harder-Narasimhan filtration in Corollary \ref{mcor} and construct the moduli space of $\nu_\alpha$-stable $\Dd$-logarithmic co-Higgs triples, using Simpson's idea and quiver interpretation. We always work in cases in which there are non-trivial co-Higgs fields; so in case of dimension one we only consider projective lines and elliptic curves. We call $\nu _\alpha$-stability with $\alpha \in \RR _{>0}$, the notion of stability for holomorphic triples. In some cases we prove that the only $\nu _\alpha$-stable holomorphic triples are obtained in a standard way from the same holomorphic triple taking the zero co-Higgs field (see Remark \ref{tre}).

It is certain that a logarithmic co-Higgs field is different from a map $\Ee \rightarrow \Ee \otimes T_X(-D)$, unless $X$ is a curve. We have a glimpse of this map in Section \ref{divisor} for the cases $X=\PP^2$ or $\PP^1 \times \PP^1$. On the contrary, in Section \ref{exttt} we consider a map $\Ee \rightarrow \Ee \otimes T_X(kD)$ with $k>0$, called a meromorphic co-Higgs field, and describe semistable meromorphic co-Higgs bundles on $\PP^1$. 

The second author would like to thank U.~Bruzzo, N.~Nitsure and L.~Brambila-Paz for many suggestions and interesting discussion.


\section{Definitions and Examples}\label{exp}

Let $X$ be a smooth complex projective variety of dimension $n\ge2$ with the tangent bundle $T_X$. For a fixed ample line bundle $\Oo_X(1)$ and a coherent sheaf $\Ee$ on $X$, we denote $\Ee \otimes \Oo_X(t)$ by $\Ee(t)$ for $t\in \ZZ$. The dimension of cohomology group $H^i(X, \Ee)$ is denoted by $h^i(X,\Ee)$ and we will skip $X$ in the notation, if there is no confusion. For two coherent sheaves $\Ee$ and $\Ff$ on $X$, the dimension of $\Ext_X^1(\Ee, \Ff)$ is denoted by $\mathrm{ext}_X^1(\Ee, \Ff)$. 

To an {\it arrangement} $\Dd=\{D_1, \ldots, D_m\}$ of smooth irreducible divisors $D_i$'s on $X$ such that $D_i\ne D_j$ for $i\ne j$, we can associate the sheaf $T_X(-\log \Dd)$ of logarithmic vector fields along $\Dd$, i.e. it is the subsheaf of the tangent bundle $T_X$ whose section consists of vector fields tangent to $\Dd$. We always assume that $\Dd$ has simple normal crossings and so $T_X(-\log \Dd)$ is locally free. It also fits into the exact sequence \cite{D}

\begin{equation}\label{log1}
0\to T_X(-\log \Dd) \to T_X \to \oplus_{i=1}^m {\epsilon_i}_*\Oo_{D_i}(D_i) \to 0,
\end{equation}
where $\epsilon_i: D_i \rightarrow X$ is the embedding. 

\begin{definition}
A {\it $\Dd$-logarithmic co-Higgs} bundle on $X$ is a pair $(\Ee, \Phi)$ where $\Ee$ is a holomorphic vector bundle on $X$ and $\Phi: \Ee \rightarrow \Ee \otimes T_X(-\log \Dd)$ with $\Phi \wedge \Phi=0$. Here $\Phi$ is called the {\it logarithmic co-Higgs field} of $(\Ee, \Phi)$ and the condition $\Phi \wedge \Phi=0$ is called the {\it integrability}. 
\end{definition}

We say that the co-Higgs field $\Phi$ is \emph{$2$-nilpotent} if $\Phi$ is non-trivial and $\Phi \circ \Phi =0$. Note that any $2$-nilpotent map $\Phi : \Ee \rightarrow \Ee \otimes T_X(-\log \Dd)$ satisfies $\Phi \wedge \Phi =0$ and so it is a non-zero co-Higgs structure on $\Ee$, i.e. a nilpotent co-Higgs structure.

Note that if $\Dd$ is empty, then we get a usual notion of co-Higgs bundle. Indeed for each $\Dd$-logarithmic co-Higgs bundle we may consider a usual co-Higgs bundle by compositing the injection in (\ref{log1}):
$$\Ee \to\Ee\otimes T_X(-\log \Dd) \to \Ee \otimes T_X.$$
Conversely, for a usual co-Higgs bundle $(\Ee, \Phi)$ we may composite the surjection in (\ref{log1}) to have a map $\Ee \rightarrow \oplus_{i=1}^m \Ee\otimes \Oo_D(D_i)$, whose vanishing would produce a logarithmic co-Higgs structure $\Ee \rightarrow \Ee \otimes T_X(-\log \Dd)$. Thus our notion of logarithmic co-Higgs bundle capture the notion of a co-Higgs field $\Phi :\Ee \rightarrow \Ee \otimes T_X$ vanishing in the normal direction to the divisors in the support of $\Dd$; in general it would not be asking for a map $\phi : \Ee \rightarrow \Ee \otimes T_X(-D)$ when $\Dd=\{D\}$. If $\dim (X)=1$, then we have $T_X(-\log \Dd ) \cong T_X(-D)$. In Section \ref{divisor} we consider a few cases in which we take $T_X(-D)$ with $D$ smooth, instead of $T_X(-\log \Dd)$.

\begin{definition}\label{ss1}
For a fixed ample line bundle $\Hh$ on $X$, a $\Dd$-logarithmic co-Higgs bundle $(\Ee, \Phi)$ is {\it $\Hh$-semistable} (resp. {\it $\Hh$-stable}) if
$$\mu(\Ff) \le (\text{resp.}<)~ \mu(\Ee)$$
for every coherent subsheaf $0\subsetneq \Ff \subsetneq \Ee$ with $\Phi(\Ff) \subset \Ff \otimes T_X(-\log \Dd)$. Recall that the slope $\mu(\Ee)$ of a torsion-free sheaf $\Ee$ on $X$ is defined to be $\mu(\Ee):=\deg (\Ee)/\rk \Ee$, where $\deg (\Ee)=c_1(\Ee)\cdot \Hh^{n-1}$. In case $\Hh\cong \Oo_X(1)$ we simply call it semistable (resp. stable) without specifying $\Hh$. 
\end{definition}

\begin{remark}
Let $(\Ee, \Phi)$ be a semistable $\Dd$-logarithmic co-Higgs bundle. For a subsheaf $\Ff \subset \Ee$ with $\Phi (\Ff)\subseteq \Ff\otimes T_X$, we have
$$\Ff \otimes T_X(-\log \Dd) =\left ( \Ff\otimes T_X\right) \cap \left ( \Ee\otimes T_X(-\log \Dd)\right )$$
and $\mathrm{Im} (\Phi)\subseteq \Ee\otimes T_X(-\log \Dd)$. Thus we get $\Phi (\Ff)\subseteq \Ff \otimes T_X(-\log \Dd)$ and so $(\Ee, \Phi)$ is semistable as a usual co-Higgs bundle. 
\end{remark}

Let us denote by $\mathbf{M}_{\Dd, X}(\chi(t))$ the moduli space of semistable $\Dd$-logarithmic co-Higgs bundles with Hilbert polynomial $\chi(t)$. It exists as a closed subscheme of $\mathbf{M}_X(\chi(t))$ the moduli space of semistable co-Higgs bundles with the same Hilbert polynomial, since the vanishing of co-Higgs fields in the normal direction to $\Dd$ is a closed condition. We also denote by $\mathbf{M}^{\circ}_{\Dd, X}(\chi(t))$ the subscheme consisting of stable ones. 

\begin{example}\label{bbb+1}
Let $X=\PP^1$ and $\Dd=\{p_1, \ldots, p_m\}$ be a set of $m$ distinct points on $X$. Then we have $T_{\PP^1}(-\log \Dd)\cong \Oo_{\PP^1}(2-m)$. Let $\Ee \cong \oplus_{i=1}^r\Oo_{\PP^1}(a_i)$ be a vector bundle of rank $r \ge 2$ on $\PP^1$ with $a_1\ge \cdots \ge a_r$ and $(\Ee, \Phi)$ be a semistable $\Dd$-logarithmic co-Higgs bundle, i.e. $\Phi: \Ee \rightarrow \Ee(2-m)$. If $a_1=\cdots =a_r$, then the pair $(\Ee, \Phi )$ is semistable for any $\Phi$. If $m\ge 3$, then $\Oo_{\PP^1}(a_1)$ would contradict the semistability of $(\Ee, \Phi)$, unless $a_1=\cdots = a_r$. If $a_1=\cdots =a_r$ and $m\ge 3$, then we have $\Phi =0$ and so $(\Ee ,\Phi)$ is strictly semistable. Assume now that $m\in \{0,1,2\}$ and then the corresponding moduli space $\mathbf{M}_{\Dd, \PP^1}(rt+d)$ is projective and $\mathbf{M}^{\circ}_{\Dd, \PP^1}(rt+d)$ is smooth with dimension $(2-m)r^2+1$, where $d=r+\sum_{i=1}^m a_i$ by \cite{Nitsure}. The case $m=0$ is dealt in \cite[Theorem 6.1]{R2}. Now assume $m=1$. Adapting the proof of \cite[Theorem 6.1]{R2}, we get Proposition \ref{bbb+2} which says in the case $\ell =-1$ that the existence of a map $\Phi$ with $(\Ee ,\Phi )$ semistable implies that $a_i\le a_{i+1}+1$ for all $i$, while conversely, if $a_i\le a_{i+1}+1$ for all $i$, then there is a map $\Phi$ with $(\Ee, \Phi )$ stable and the set of all such $\Phi$ is a non-empty open subset of the vector space $H^0(\mathcal{E}nd (\Ee )(1))$.

Now assume that $m=2$ and so $\Phi \in \mathrm{End}(\Ee)$. If $a_1=\cdots = a_r$, then $\Phi$ is given by an $(r\times r)$-matrix of constants. Since the matrix has an eigenvector, the pair $(\Ee ,\Phi)$ is strictly semistable for any $\Phi$. Now assume $a_1>a_r$ and let $h$ be the maximal integer $i$ with $a_i=a_1$. Write $\Ee \cong\Ff \oplus \Gg$ with $\Ff := \oplus _{i=1}^{h} \Oo _{\PP^1}(a_i)$ and $\Gg:= \oplus _{i=h+1}^r \Oo _{\PP^1}(a_i)$. Since any map $\Ff \rightarrow \Gg$ is the zero map, we have $\Phi (\Ff )\subseteq \Ff$ for any $\Phi : \Ee \rightarrow \Ee$ and so $(\Ee ,\Phi)$ is not semistable.
\end{example}

\subsection{Projective spaces}
In \cite{BH} we introduce a simple way of constructing nilpotent co-Higgs sheaves $(\Ee, \Phi)$ of rank $r\ge 2$, fitting into the exact sequence
\begin{equation}\label{eeqb}
0 \to \Oo _X^{\oplus (r-1)}\to \Ee \to \Ii _Z\otimes \Aa \to 0
\end{equation}
for a two-codimensional locally complete intersection $Z\subset X$ and $\Aa \in \Pic (X)$ such that $H^0(T_X\otimes \Aa^\vee)\ne 0$. We replace $T_X$ by $T_X(-\log \Dd)$ for a simple normal crossing divisor $\Dd$ in (\ref{eeqb}) to obtain $2$-nilpotent $\Dd$-logarithmic co-Higgs sheaves. 

\begin{example}\label{eex}
Let $X=\PP^n$ with $n\ge 2$ and take $\Dd=\{D_1, \ldots, D_m\}$ with $D_i\in |\Oo_{\PP^n}(1)|$. If $1\le  m \le n$, we have $T_{\PP^n}(-\log \Dd) \cong \Oo_{\PP^n}^{\oplus (m-1)} \oplus \Oo_{\PP^n}(1)^{\oplus (n-m+1)}$ by \cite{DK}, and in particular we have $h^0(T_{\PP^n}(-\log \Dd)(-1))>0$. Thus we may apply the proof of \cite[Theorem 1.1]{BH} to get the following: here the invariant $x_{\Ee}$ is defined to be the maximal integer $x$ such that $h^0(\Ee(-x))\ne 0$. 

\begin{proposition}\label{ttts}
The set of nilpotent maps $\Phi : \Ee \rightarrow \Ee \otimes T_{\PP^n}(-\log \Dd)$ on a fixed stable reflexive sheaf $\Ee$ of rank two on $\PP^n$ is an $(n-m+1)$-dimensional vector space only if $c_1(\Ee)+2x_{\Ee}=-3$. In the other cases the set is trivial. 
\end{proposition}
\end{example}

\begin{remark}
Consider the case $m=n+1$ in Example \ref{eex} with $\cup_{i=1}^{n+1}D_i = \emptyset$. Then we have $T_{\PP^n}(-\log \Dd )\cong \Oo _{\PP^n}^{\oplus n}$. Let $\Ee$ be a reflexive sheaf of rank $r\ge 2$ on $\PP^n$ with a semistable (resp. stable) logarithmic co-Higgs structure $(\Ee,\Phi)$. Note that if $\Phi$ is trivial, the semistability (resp. stability) of $(\Ee ,\Phi)$ is equivalent to the semistability (resp. stability) of $\Ee$. Now assume $\Phi \ne 0$. Since $T_{\PP^n}(-\log \Dd )\cong \Oo _{\PP^n}^{\oplus n}$, $\Ee$ is not simple and in particular it is not stable. We claim that $\Ee$ is semistable. If not, call $\Gg$ be the first step of the Harder-Narasimhan filtration of $\Ee$. By a property of the Harder-Narisimhan filtration there is no non-zero map $\Gg \rightarrow \Ee /\Gg$ and so no non-zero map $\Gg \rightarrow ( \Ee /\Gg)\otimes T_{\PP^n}(-\log \Dd)$. Thus we get $\Phi (\Gg )\subseteq \Gg \otimes T_{\PP^n}(-\log \Dd )$, contradicting the semistability of $(\Ee, \Phi)$. Now assume $n=2$ and take $\Aa \cong \Oo _{\PP^2}$ in (\ref{eeqb}) with $\deg (Z)\ge r-1$. Then we get many strictly semistable and indecomposable vector bundles $\Ee$ with $\Phi \ne 0$ and $2$-nilpotent.
\end{remark}


\begin{example}
Let $X=\PP^2$ and take $\Dd = \{D\}$ with $D$ a smooth conic. Since $h^0(T_{\PP^2}) =8$ and $h^0(\Oo _D(D)) =h^0(\Oo _D(2)) =5$, we have $h^0(T_{\PP^2}(-\log \Dd )) >0$ from (\ref{log1}). By taking $\Aa \cong\Oo _{\PP^2}$ in \cite[Equation (1) of Condition 2.2]{BH}, we get a strictly semistable logarithmic co-Higgs bundle $(\Ee ,\Phi)$ with a non-zero co-Higgs field $\Phi$, where $\Ee$ is strictly semistable of any arbitrary rank $r\ge 2$ with any non-negative integer $c_2(\Ee )=\deg (Z)$. Moreover, for any integer $c_2(\Ee)\ge r-1$ we may find an indecomposable one.
\end{example}

\begin{example}
Let $X\subset \PP^{n+1}$ be a smooth quadric hypersurface. Let $D \subset X$ be a smooth hyperplane section of $X$ with $H\subset \PP^{n+1}$ the hyperplane such that $D=X \cap H$ and take $\Dd=\{D\}$. If $p\in \PP^{n+1}$ is the point associated to $H$ by the isomorphism between $\PP^{n+1}$ and its dual induced by an equation of $X$, then we have $p\notin X$ since $X$ is smooth. Letting $\pi_p: X\rightarrow \PP^n$ denote the linear projection from $p$, we have $T_X(-\log \Dd )\cong \pi_p ^\ast (\Omega ^1_{\PP^n}(2))$ by \cite[Corollary 4.6]{BHlog}. Since $\Omega ^1_{\PP^n}(2)$ is globally generated, so is $T_X(-\log \Dd )$ and in particular $H^0(T_X(-\log \Dd)) \ne 0$. By taking $\Aa \cong \Oo _X$ in \cite[Equation (1) of Condition 2.2]{BH}, we get a strictly semistable logarithmic co-Higgs bundle $(\Ee ,\Phi)$ with a non-zero co-Higgs field $\Phi$, where $\Ee$ is strictly semistable of any arbitrary rank $r\ge 2$.
\end{example}

\subsection{Smooth quadric surfaces}\label{quad}
Let $X = \PP^1\times \PP^1$ be a smooth quadric surface and we may assume for a vector bundle $\Ee$ of rank two that 
$$\det (\Ee )\in \{\Oo _X,\Oo _X(-1,0),\Oo _X(0,-1),\Oo _X(-1,-1)\}.$$
The case of the usual co-Higgs bundle with $\Dd=\emptyset$ is done in \cite[Theorem 4.3]{VC1}. We assume either
\begin{itemize}
\item [(i)] $\Dd \in \left \{|\Oo _X(1,0)|,|\Oo _X(2,0)|,|\Oo _X(0,1)|,|\Oo _X(0,2)|\right\}$, or
\item [(ii)]$\Dd =L\cup R$ with $L\in |\Oo _X(1,0)|$ and $R\in |\Oo _X(0,1)|$.
\end{itemize}
In the latter case $T_X(-\log \Dd)$ fits into the exact sequence
\begin{equation}\label{eqv2}
0 \to T_X(-\log \Dd )\to \Oo _X(2,0)\oplus \Oo _X(0,2)\to \Oo _L \oplus \Oo _R\to 0,
\end{equation}
because $\Oo _L(L)\cong \Oo _L$, $\Oo _R({R})\cong \Oo _R$ and $T_X\cong \Oo_X(2,0)\oplus \Oo_X(0,2)$. In particular, we have $h^0(T_X(-\log  \Dd )(i,j)) >0$ for all $(i,j)\in \{(0,0), (-1,0), (0,-1)\}$. We may also consider the following cases:
\begin{itemize}
\item [(iii)] $\Dd = L\cup L'\cup R$ with $L, R$ as above and $L\ne L'\in |\Oo _X(1,0)|$; we still have $h^0(T_X(-\log \Dd )(i,j)) >0$ for $(i,j)\in \{ (0,0), (0,-1)\}$. 
\item [(iv)] $\Dd =L\cup L'\cup R\cup R'$ with $L$, $L'$, $R$ as above and $R\ne R'\in |\Oo _X(0,1)|$.
\end{itemize}
Indeed, if $\Dd$ consists of $a$ lines in $|\Oo_X(1,0)|$ and $b$ lines in $|\Oo_X(0,1)|$, then we have $T_X(-\log \Dd)\cong \Oo_X(2-a,0)\oplus \Oo_X(0,2-b)$ by \cite[Proposition 6.2]{BHlog}. 

Assume that $\Ee$ fits into the following exact sequence as in \cite[Equation (3.1)]{VC1} 
\begin{equation}\label{eqv1}
0\to \Oo _X(r,d)\to \Ee \to \Oo _X(r',d')\otimes \Ii _Z\to 0,
\end{equation}
where $Z\subset X$ is a zero-dimensional scheme, $\det (\Ee ) \cong \Oo _X(r+r',d+d'')$ and $c_2(\Ee )=\deg (Z)+rd' +r'd$. Note that that any logarithmic co-Higgs bundle is co-Higgs in the usual sense and so the set of all $(c_1,c_2)$ allowed for $\Dd$ is contained in the one allowed for $\Dd =\emptyset$. In particular, if we are concerned only in $\Oo _X(1,1)$-semistability, the possible pairs $(c_1,c_2)$ are contained in the one described in \cite[Theorem 4.3]{VC1}. Moreover, any existence for the case $\Dd =L\cup R$ implies the existence for $\Dd \in \{|\Oo _X(1,0)|,\Oo _X(0,1)|\}$. 

\quad (a) First assume $\det (\Ee )\cong \Oo _X$ and we prove the existence for $c_2\ge 0$. In this case we take $r=d=r'=d' =0$ and the $2$-nilpotent co-Higgs structure induced by $\Ii _Z \rightarrow T_X(-\log \Dd )$, i.e. by a non-zero section of $T_X(-\log \Dd)$. This construction gives $(\Ee ,\Phi)$ with $\Ee$ strictly semistable for any polarization.

\quad (b) Assume $\det (\Ee )\cong \Oo _X(-1,0)$ by symmetry and see the existence for $c_2\ge 0$. In case $h^0(T_X(-\log \Dd )(-1,0)) >0$, we take $(r,r',d,d')=(-1,0,0,0)$ and $\Phi$ induced by a non-zero map $\Ii _Z \rightarrow T_X(-\log \Dd)(-1,0)$. Then $\Ee$ is stable for every polarization, unless $Z=\emptyset$ and $\Ee$ splits, because $Z\ne \emptyset$ would imply $h^0(\Ee )=0$; even when $Z=\emptyset$ and so $\Ee \cong \Oo _X\oplus \Oo _X(-1,0)$, the pair $(\Ee ,\Phi)$ is stable for every polarization.

\quad (c) Assume $\det (\Ee)\cong \Oo _X(-1,-1)$ and take $(r,d)=(-1,0)$ and $(r', d')=(0,-1)$ with $\Dd \in |\Oo _X(1,0)|$. Then we have $h^0(T_X(-\log \Dd )(-1,1)) >0$ and $c_2(\Ee ) =\deg (Z)+1$. We get that $\Ee$ is semistable with respect to $\Oo _X(1,1)$.

\section{Existence}

\begin{proposition}\label{aa1}
Assume $\dim (X)=2$ and let $\Dd \subset X$ be a simple normal crossing divisor. For fixed $\Ll \in \mathrm{Pic}(X)$ and an integer $r\ge 2$, there exists an integer $n=n_{X,\Dd}(\Ll, r)$ such that for all integers $c_2\ge n$ there is a $2$-nilpotent $\Dd$-logarithmic co-Higgs bundle $(\Ee ,\Phi)$ with $\Phi \ne 0$, where $\Ee$ is an indecomposable vector bundle of rank $r$ with Chern classes $c_1(\Ee)\cong \Ll$ and $c_2(\Ee)=c_2$. 
\end{proposition}

\begin{proof}
Fix a very ample $\Rr\in \Pic (X)$ such that 
\begin{itemize}
\item $h^0(\omega _X\otimes (\Ll^{\otimes (r-1)} \otimes \Rr^{\otimes r})^\vee )= 0$;
\item $h^0(T_X(-\log \Dd)\otimes \Ll ^{\otimes (r-1)}\otimes \Rr ^{\otimes r})>0$;
\item $\Ll \otimes \Rr$ is spanned. 
\end{itemize}
Set 
$$n=n_{X,\Dd}(r,\Ll):=r-(r-1)(r-2)\Ll^2 - (r-1)^2 \Rr^2-(2r-3)(r-1)\Ll {\cdot}\Rr.$$
For each $c_2\ge n$, let $S\subset X$ be a union of general $(c_2 +r-n)$ points and consider a general extension
$$0\to (\Ll \otimes \Rr)^{\oplus (r-1)}\to \Ee \to \Ii _S\otimes (\Ll^{\otimes (r-2)}\otimes \Rr ^{\otimes (r-1)})^\vee \to 0.$$
From the choice of $\Rr$ the Cayley-Bacharach condition is satisfied and so $\Ee$ is locally free with $c_1(\Ee)\cong \Ll$ and $c_2(\Ee )=c_2$. Now from a non-zero section in $H^0(T_X(-\log \Dd)\otimes \Ll^{\otimes (r-1)} \otimes \Rr ^{\otimes r})$ we have a non-zero map $\phi : \Ii _S\otimes (\Ll^{\otimes (r-2)}\otimes \Rr ^{\otimes (r-1)})^\vee \rightarrow \Ll \otimes \Rr \otimes T_X(-\log \Dd )$, inducing a non-zero map $\Phi : \Ee \rightarrow \Ee \otimes T_X(-\log \Dd )$ that is $2$-nilpotent and so integrable. 

Thus to complete the proof it is sufficient to prove that $\Ee$ is indecomposable for a suitable $\Rr$. Assume $\Ee \cong \Ee _1\oplus \cdots \oplus \Ee _k$ with $k\ge 2$ and each $\Ee_i$ indecomposable and locally free of positive rank. Since $\Rr$ is very ample and $\Ll \otimes \Rr$ is spanned, the image of the evaluation map $H^0(\Ee)\otimes \Oo _X\rightarrow \Ee$ is isomorphic to $(\Ll \otimes \Rr)^{\oplus (r-1)}$ and its cokernel is isomorphic to $ \Ii _S\otimes (\Ll^{\otimes (r-2)}\otimes \Rr ^{\otimes (r-1)})^\vee$. Thus, up to a permutation of the factors, we have $(\Ll \otimes \Rr)^{\oplus (r-1)}\cong \Ee _1\oplus \cdots \oplus \Ee _{k-1}\oplus \Ff$ with $\Ff$ a vector bundle and $\Ee _k/\Ff \cong  \Ii _S\otimes (\Ll^{\otimes (r-2)}\otimes \Rr ^{\otimes (r-1)})^\vee$. Since $\Ee _1$ is indecomposable, we get that $\Ee _1\cong \Ll\otimes \Rr$. But since $\sharp (S) \ge r$, we have $\mathrm{ext}_X^1( \Ii _S\otimes (\Ll^{\otimes (r-2)}\otimes \Rr ^{\otimes (r-1)})^\vee ,\Oo _X)\ge r$ and so we may choose $\Ee$ so that $\Ll \otimes \Rr $ is not a factor of $\Ee$.
\end{proof}

\begin{proposition}\label{aa2}
Assume $n=\dim (X)\ge 3$ and let $\Dd\subset X$ be a simple normal crossing divisor. For a fixed $\Ll \in \mathrm{Pic}(X)$ and an integer $r\ge n$, there exists a $2$-nilpotent $\Dd$-logarithmic co-Higgs bundle $(\Ee ,\Phi)$, where $\Ee$ is an indecomposable vector bundle of rank $r$ on $X$ with $\det (\Ee )\cong \Ll$.
\end{proposition}

\begin{proof}
We first assume that $\Ll ^\vee$ is very ample with
\begin{itemize}
\item $h^1(\Ll ^\vee )=h^2(\Ll ^\vee )=0$, where we use the assumption $n\ge 3$;
\item $h^0(\Ll ^\vee )\ge r-1$ and $h^0(\Ll ^\vee \otimes T_X(-\log \Dd )) > 0$.
\end{itemize}
Fix a very ample line bundle $\Hh$ on $X$ such that $h^0(\Hh ^\vee \otimes \Ll ^\vee )=h^1((\Hh ^\vee )^{\otimes 2}\otimes \Ll ^\vee )=0$, e.g. by taking $\Hh \cong (\Ll ^\vee )^{\otimes 2}$ and applying Kodaira's vanishing. Let $Y\subset X$ be a general complete intersection of two elements of $|\Hh |$ and then $Y$ is a non-empty connected manifold of codimension $2$ with normal bundle $N_Y$, isomorphic to $\Hh_{|Y} ^{\oplus 2}$. The line bundle $\Rr := \wedge ^2N_Y\otimes \Ll ^\vee _{|Y} \cong (\Hh ^{\otimes 2}\otimes \Ll ^\vee )_{|Y}$ is a very ample line bundle on $Y$ and we have $h^0(Y,\Rr )\ge h^0(Y,(\Ll ^\vee )_{|Y})$. From the exact sequence 
$$0\to (\Hh ^\vee )^{\otimes 2}\to (\Hh ^\vee )^{\oplus 2}\to \Ii _Y\to 0$$ 
we get $h^0(\Ii _Y\otimes \Ll^\vee ) =0$ and so $h^0(Y,\Rr )\ge h^0(Y,(\Ll ^\vee )_{|Y}) \ge r-1$. Since $\Rr$ is spanned and $\dim (Y)=n-2$, a general $(n-1)$-dimensional linear subspace $V\subset H^0(Y,\Rr)$ spans $\Rr$. Hence there are linearly independent sections $s_1,\dots ,s_{r-1}$ of $H^0(Y,\Rr)$ spanning $\Rr$. Since $H^2(\Ll ^\vee )=0$, by the Hartshorne-Serre correspondence the sections $s_1,\dots ,s_{r-1}$ give a vector bundle $\Ee$ of rank $r$ fitting into an exact sequence (see \cite[Theorem 1.1]{Arrondo})
$$0\to \Oo _X^{\oplus (r-1)} \to \Ee \to \Ii _Y\otimes \Ll \to 0.$$
In particular we have $\det (\Ee )\cong \Ll$. Any non-zero section of $H^0(\Ll ^\vee \otimes T_X(-\log \Dd ))$ gives a $2$-nilpotent logarithmic co-Higgs structures on $\Ee$ with $\Phi \ne 0$. Now it remains to show that $\Ee$ is indecomposable. Assume $\Ee \cong \Gg_1 \oplus \Gg_2$ with $\Gg_i$ non-zero.  Let $\Gg_i'$ be the image of the evaluation map $H^0(\Gg_i )\otimes \Oo _X\rightarrow \Gg_i$ for $i=1,2$. Since $\Ll ^\vee$ is very ample, we have $h^0(\Ee )=r-1$ and the image of the evaluation map $H^0(\Ee )\otimes \Oo _X\rightarrow \Ee$ is isomorphic to $\Oo _X^{\oplus (r-1)}$ and so $\Gg_1'\oplus \Gg _2'\cong \Oo _X^{\oplus (r-1)}$. In particular, we have $\Gg_i \cong \Gg_i'$ for some $i$ and so at least one of the factors of $\Ee$ is trivial. Set $\Ee \cong \Oo _X\oplus \Ff$ with $\mathrm{rank}(\Ff )=r-1$. By \cite[Theorem 1.1]{Arrondo} the bundle $\Ff$ comes from $u_1,\dots ,u_{r-2}\in H^0(Y,\Rr)$ and so $\Ee$ is induced by the sections $u_1,\dots ,u_{r-2},0$. Since  $H^1(\Ll ^\vee )=0$, the uniqueness part of \cite[Theorem 1.1]{Arrondo} gives that $s_1,\dots ,s_{r-1}$ generate the linear subspace of $H^0(Y,\Rr)$ spanned by $u_1,\dots ,u_{r-2}$ and so they are not linearly independent, a contradiction.

Now we drop any assumption on $\Ll$. Take an integer $m\gg 0$ and set $\Ll' := \Ll \otimes (\Hh ^\vee )^{\otimes (mr)}$. Then we get that $(\Ll') ^\vee $ is very ample and $H^2((\Ll')^\vee )=0$. By the first part there is $(\Ee' ,\Phi' )$ with $\det (\Ee')\cong \Ll'$. We may take $\Ee := \Ee' \otimes (\Hh ^{\vee})^{\otimes m}$ and let $\Phi : \Ee \rightarrow \Ee \otimes T_X(-\log \Dd )$ be the non-zero map induced by $\Phi'$.
\end{proof}

Allowing non-locally free sheaves, we may extend  Proposition \ref{aa2} to all ranks at least two in the following way.

\begin{proposition}\label{aa2.00}
Under the same assumption as in Proposition \ref{aa2} with $2\le r\le n-1$, there exists a $2$-nilpotent $\Dd$-logarithmic co-Higgs reflexive sheaf $(\Ee ,\Phi)$, where $\Ee$ is indecomposable of rank $r$ with $\det (\Ee )\cong \Ll$ and non-locally free locus of dimension at most $(n-r-1)$.
\end{proposition}

\begin{proof}
We follow the proof of Proposition \ref{aa2}. We first assume that $\Ll ^\vee$ is very ample and take $(\Hh, Y, \Rr)$ as in the proof of Proposition \ref{aa2}. Since $r\ge 2$, we may find $r-1$ elements $s_1,\dots ,s_{r-1}\in H^0(Y,\Rr)$ spanning $\Rr$ outside a subset $T$ of $Y$ with $\dim (T) \le \dim (Y)-r+1 = n-r-1$. By \cite{Hartshorne1}, the sections $s_1,\dots ,s_{r-1}$ give a reflexive sheaf $\Ee$ of rank $r$ on $X$ with $\det (\Ee) \cong \Ll$ and $\Ee$ locally free outside $T$.

The reduction to the case in which $\Ll^\vee$ is very ample can be done, using the argument in the proof of Proposition \ref{aa2}.
\end{proof}







\section{Vanishing along divisors}\label{divisor}
As observed, the notion of logarithmic co-Higgs bundle is not asking for a map $\phi : \Ee \rightarrow \Ee \otimes T_X(-D)$ if $\dim (X)\ge 2$. In this section we study vector bundles of rank two on a projective plane and a smooth quadric surface with sections in $H^0(\mathcal{E}nd(\Ee)\otimes T_X(-D))$. 

\subsection{Projective plane}
Let $X = \PP^2$ and take $D\in |\Oo_{\PP^2}(1)|$ a projective line. Then we have $T_{\PP^2}(-D) =T_{\PP^2}(-1)$ and so $h^0(T_{\PP^2}(-D)) =3$. We may give a $2$-nilpotent co-Higgs structure on a vector bundle $\Ee$ of rank $2$ fitting into the exact sequence
\begin{equation}\label{++}
0 \to \Oo _{\PP^2}\to \Ee \to \Ii _Z\to 0
\end{equation}
from a non-zero section in $H^0(T_{\PP^2}(-D))$. Thus there exists a strictly semistable co-Higgs bundle of rank two for all $c_2\ge 0$, which is indecomposable for $c_2>0$. Indeed for any such bundles with positive $c_2$ we have a three-dimensional vector space of $2$-nilpotent co-Higgs structures. On the contrary we have some results on non-existence of co-Higgs bundles on projective spaces in \cite[Section 3]{BH}. Applying the same argument to $T_{\PP^n}(-1)$, we get the following, as in Proposition \ref{ttts}

\begin{proposition}
If $\Ee$ is a stable reflexive sheaf of rank two on $\PP^n$ with $n\ge 2$, then any nilpotent map $\Phi : \Ee \rightarrow \Ee \otimes T_{\PP^n}(-1)$ is trivial. 
\end{proposition}

\subsection{Quadric surface}
Let $X=\PP^1 \times \PP^1$ and take $D\in |\Oo _X(1,0)|$; by symmetry the case $D\in |\Oo _X(0,1)|$ is similar. We have $T_X(-D)\cong \Oo _X(1,0)\oplus \Oo _X(-1,2)$.

\quad (a) In case $\det (\Ee )\cong \Oo _X$ we prove the existence for $c_2\ge 0$. By taking$r=d=r'=d' =0$, we obtain a $2$-nilpotent co-Higgs structure induced by $\Ii _Z \rightarrow T_X(-D )$, i.e. by a non-zero section of $T_X(-D)$. This construction gives $(\Ee ,\Phi)$ with $\Ee$ strictly semistable for any polarization.

\quad (b) In case $\det (\Ee )\cong \Oo _X(-1,0)$ we also see the existence for $c_2\ge 0$. Since $h^0(T_X(-D)(-1,0)) >0$, we take $(r,r',d,d')=(-1,0,0,0)$ and $\Phi$ induced by a non-zero map $\Ii _Z \rightarrow T_X(-D)(-1,0)$. Then $\Ee$ is stable for every polarization, unless $Z=\emptyset$ and $\Ee$ splits, because $Z\ne \emptyset$ would imply $h^0(\Ee)=0$; even when $Z=\emptyset$ and so $\Ee \cong \Oo _X\oplus \Oo _X(-1,0)$, the pair $(\Ee ,\Phi)$ is stable for every polarization.

\quad (c) Assume $\det (\Ee)\cong \Oo _X(-1,-1)$ and take $(r,d)=(-1,0)$ and $(r', d')=(0,-1)$ with $D\in |\Oo _X(1,0)|$. Note that $h^0(T_X(-D )(-1,1)) >0$ and $c_2(\Ee ) =\deg (Z)+1$. Then we get that $\Ee$ is semistable with respect to $\Oo _X(1,1)$.

\begin{remark}
\begin{enumerate}
\item It is likely that we may not apply our method of construction of $2$-nilpotent co-Higgs structure to the case when $\det (\Ee) \cong \Oo_X(0,-1)$, because it requires a non-zero section in $h^0(T_X(-D)(-1,0))$, which is trivial. 
\item Take $\Dd= L\cup R$ with $L, R\in  |\Oo _X(1,0)|$ and $L\ne R$; the case with $L,R\in |\Oo _X(0,1)|$ is similar. Then the existence for the case $c_1(\Ee)=\Oo_X(0,0)$ can be done for any $c_2\ge 0$ as above.
\end{enumerate}
\end{remark}


\section{Extension of co-Higgs bundles}\label{exttt}

Fix an ample line bundle $\Hh$ on $X$ and a vector bundle $\Gg$. Then we may define $\Hh$-(semi)stability for a pair $(\Ee, \Phi)$ with $\Ee$ a torsion-free sheaf and $\Phi: \Ee \rightarrow \Ee \otimes \Gg$, similarly as in Definition \ref{ss1} with $\Gg$ instead of $T_X(-\log \Dd)$. Then the definition of (logarithmic) co-Higgs bundle is obtained by taking $\Gg\in \{T_X,T_X(-\log \Dd ), T_X(-D)\}$ with the integrability condition $\Phi \wedge\Phi =0$. Note that it is enough to check the integrability condition on a non-empty open subset $U$ of $X$. 

\begin{definition}
Fix an effective divisor $D\subset X$ and a positive integer $k$, for which we take $\Gg := T_X(kD)$. A pair $(\Ee, \Phi)$ is called a {\it meromorphic co-Higgs sheaf} with poles of order at most $k$ contained in $\Dd$, if it satisfies the integrability condition on $U:=X \setminus D$. 
\end{definition}

Via the inclusion $T_X\hookrightarrow T_X(kD)$ induced by a section of $\Oo _X(kD)$ with $kD$ as its zeros, we see that any co-Higgs sheaf is also a meromorphic co-Higgs for any $k$ and $D$. A meromorphic co-Higgs sheaf with poles contained in $D$ induces an ordinary co-Higgs sheaf $(\Ff ,\phi)$ on the non-compact manifold $U$ and our definition of meromorphic co-Higgs sheaves captures the extension of $(\Ff ,\phi)$ to $X$ with at most poles on $D$ of order at most $k$. 

\begin{remark}
We may generalize the definition of a meromorphic co-Higgs sheaf as follows: take $D=\cup _{i=1}^{s} D_i$ with each $D_i$ irreducible and consider $\sum _{i=1}^{s} k_iD_i$, $k_i$ a positive integer, instead of $kD$. Then we get the co-Higgs sheaves $(\Ff ,\phi)$ on $X\setminus D$, which extends meromorphically to $X$ with poles of order at most $k_i$ on each $D_i$.
\end{remark}

Our method used in constructing $2$-nilpotent co-Higgs sheaves (see \cite[Condition 2.2]{BH}) can be applied to construct $2$-nilpotent meromorphic co-Higgs sheaves, if $h^0(T_X(kD))>0$; we may easily check when the construction gives locally free ones. In the set-up of Sections \ref{quad} and \ref{divisor} we immediately see how to construct examples filling in several Chern classes. 

Assume that $\dim (X)=1$ and let $D = p_1+\cdots +p_s$ be $s$ distinct points on $X$. Set $\ell:= \deg (\sum _{i=1}^{s} k_ip_i)$ and $r:= \mathrm{rank}(\Ee)$. We adapt the proof of \cite[Theorem 6.1]{R2} with only very minor modifications to prove the following result. To cover the case needed in Example \ref{bbb+1} we allow as $\ell$ an integer at least $-1$.

\begin{proposition}\label{bbb+2}
Let $\Ee \cong \Oo_{\PP^1}(a_1)\oplus \cdots \oplus \Oo_{\PP^1}(a_r)$ be a vector bundle of rank $r\ge 2$ on $\PP^1$ with $a_1\ge \cdots \ge a_r$.
\begin{itemize}
\item [(i)] If $(\Ee, \Phi)$ is semistable with a map $\Phi : \Ee \rightarrow \Ee (2+\ell )$, then we have $a_{i+1} \ge a_i -\ell -2$ for each $i \le r-1$.
\item [(ii)] Conversely, if $a_{i+1} \ge a_i -\ell -2$ for each $i \le r-1$, then there is a map $\Phi : \Ee \rightarrow \Ee (2+\ell )$ such that no proper subbundle $\Ff \subset \Ee$ satisfies $\Phi (\Ff )\subseteq \Ff (2+\ell)$, and in particular $(\Ee ,\Phi)$ is stable. The set of all such $\Phi$ is non-empty open subset of the vector space $H^0(\mathcal{E}nd (\Ee )(2+\ell ))$.
\end{itemize}
\end{proposition}

\begin{proof}
Assume the existence of an integer $i$ such that $a_{i+1} \le a_i -\ell -3$ and take $\Phi : \Ee \rightarrow \Ee (2+\ell )$. Set $\Ee =\Ff \oplus \Gg$ with $\Ff := \oplus _{j=1}^{i} \Oo _{\PP^1}(a_j)$ and $\Gg:= \oplus _{j=i+1}^r \Oo _{\PP^1}(a_j)$. Since any map $\Ff \rightarrow \Gg (2+\ell)$ is the zero map, we have $\Phi (\Ff )\subseteq \Ff (2+\ell)$ and so $(\Ee ,\Phi)$ is not semistable.

Now assume $a_{i+1}\ge a_i-\ell -2$ for all $i$. Write $\Phi$ as an $(r\times r)$-matrix $B$ with entries $b_{i,j}\in \mathrm{Hom}(\Oo _{\PP^1}(a_i),\Oo _{\PP^1}(a_j+2+\ell ))$. For fixed homogeneous coordinates $z_0, z_1$ on $\PP^1$ with $\infty = [1:0]$ and $0 =[0:1]$, see a homogeneous polynomial of degree $d$ in the variables $z_0,z_1$ as a polynomial of degree at most $d$ in the variable $z:= z_0/z_1$. Take 
$$B=\begin{bmatrix} 0&1&0&\cdots &z\\
0&0&1&\cdots &0\\
\vdots & \vdots & \ddots & \ddots&\vdots \\
0&0&\cdots&\cdots&1\\
0&0&\cdots&\cdots &0\end{bmatrix}$$
so that $b_{i,j} =0$ unless either $(i,j)=(1,r)$ or $j=i+1$; we take $b_{i,i+1} =1$ for all $i$, i.e. the elements of $\CC[z]$ associated to $z_1^{a_{i+1}-a_i+2+\ell}$, and $b_{1,r} =z$, the element of $\CC[z]$ associated to $z_0z_1^{a_r-a_1+1+\ell}$. Then there is no proper subbundle $\Ff\subset \Ee$ with $\Phi (\Ff )\subseteq \Ff (2+\ell)$, because the characteristic polynomial of $B$ is $\det (tI-B)= (-1)^{r-1}z+t^r$, which is irreducible in $\CC[z,t]$.
\end{proof}

\begin{remark}
Assume the genus $g$ of $X$ is at least $2$ and that $2-2g+\ell <0$. Then there exists no semistable meromorphic co-Higgs bundle $(\Ee ,\Phi)$ with $\Phi \ne 0$. Indeed, for any pair $(\Ee, \Phi)$, the map $\Phi$ would be a non-zero map between two semistable vector bundles with the target having lower slope.
\end{remark}


\section{Moduli over projective plane}\label{mex}

Let $X=\PP^n$ and fix $\Dd=\{D\}$ with $D\in |\Oo_{\PP^n}(1)|$. Then we have $T_{\PP^n}(-\log \Dd) \cong \Oo_{\PP^n}(1)^{\oplus n}$ and 
$$\Phi=(\phi_1, \ldots, \phi_n): \Ee \rightarrow \Ee \otimes T_{\PP^n}(-\log \Dd)$$
with $\phi_i : \Ee \rightarrow \Ee(1)$ for $i=1,\ldots, n$. Assume that $(\Ee, \Phi)$ is a semistable co-Higgs bundle of rank $r$ along $\Dd$. If $\Ee \cong \oplus_{i=1}^r \Oo_{\PP^n}(a_i)$ is a direct sum of line bundles on $\PP^n$ with $a_i \ge a_{i+1}$ for all $i$, then we get $a_i \le a_{i+1}+1$ for all $i$ by adapting the proof of \cite[Theorem 6.1]{R2}. Thus in case $\mathrm{rank}(\Ee)=r=2$, by a twist we fall into two cases: $\Oo_{\PP^n}^{\oplus 2}$ or $\Oo_{\PP^n}\oplus \Oo_{\PP^n}(-1)$. 

We denote by $\mathcal{E}nd_0(\Ee)$ the kernel of the trace map $\mathcal{E}nd(\Ee) \rightarrow \Oo_X$, the trace-free part, and then we have 
$$\mathcal{E}nd(\Ee)\otimes T_X(-\log \Dd) \cong (\mathcal{E}nd_0(\Ee)\otimes T_X(-\log \Dd)) \oplus T_X(-\log \Dd).$$
Thus any co-Higgs field $\Phi$ can be decomposed into $\Phi_1+ \Phi_2$ with $\Phi_1 \in H^0(\mathcal{E}nd_0(\Ee)\otimes T_X(-\log \Dd))$ and $\Phi_2\in H^0(T_X(-\log \Dd))$. Note that $(\Ee, \Phi)$ is (semi)stable if and only if $(\Ee, \Phi_1)$ is (semi)stable. Thus we may pay attention only to trace-free logarithmic co-Higgs bundles. Let us denote by $\mathbf{M}_{\Dd}(c_1, c_2)$ the moduli of semistable trace-free $\Dd$-logarithmic co-Higgs bundles of rank two on $\PP^2$ with Chern classes $(c_1, c_2)$. In case $\Dd=\emptyset$ we simply denote the moduli space by $\mathbf{M}(c_1, c_2)$. 
 
\begin{proposition}\label{de1}
$\mathbf{M}_{\Dd}(-1,0)$ is isomorphic to the total space of $\Oo_{D}(-2)^{\oplus 6}$. 
\end{proposition}
\begin{proof}
By \cite[Lemma 3.2]{hart} $\Ee$ is not semistable for $(\Ee, \Phi) \in \mathbf{M}_{\Dd}(-1,0)$ and so we get an exact sequence $0 \rightarrow  \Oo_{\PP^2}(t) \rightarrow \Ee\rightarrow  \Ii_Z(-t-1)\rightarrow 0$ with $t\ge 0$. Here $\Phi(\Oo_{\PP^2}(t))\subset \Ii_Z(-t)$ is a non-trivial subsheaf and so we get $t=0$ and $Z=\emptyset$. Thus we get $\Ee \cong \Oo_{\PP^2}\oplus \Oo_{\PP^2}(-1)$. Then following the proof of \cite[Theorem 5.2]{Rayan} verbatim, we see that 
$$\mathbf{M}_{\Dd}(-1,0) \cong H^0(\Oo_{\PP^2}(2))\times (H^0(\Oo_{\PP^2}^{\oplus 2}) \setminus \{0\})\sslash \CC^*,$$
where $\CC^*$ acts on $H^0(\Oo_{\PP^2}(2))$ with weight $-2$ and on $H^0(\Oo_{\PP^2}^{\oplus 2})\setminus \{0\}$ with weight $1$. Thus we get that $\mathbf{M}_{\Dd}(-1,0)$ is isomorphic to the total space of $\Oo_{\PP^1}(-2)^{\oplus 2}$. Indeed, from the sequence (\ref{log1}) twisted by $-1$, we can identify $\PP H^0(\Oo_{\PP^2}^{\oplus 2})$ with $D$ and so $\mathbf{M}_{\Dd}(-1,0)$ can be obtained by restricting $\Oo_{\PP^2}(-2)^{\oplus 6}$ to $D$ as a closed subscheme of $\mathbf{M}(-1,0)$, which is isomorphic to the total space of $\Oo_{\PP^2}(-2)^{\oplus 6}$ (see \cite[Theorem 5.2]{Rayan}). 
\end{proof}

Recall in \cite[Page 1447]{Rayan} that $\mathbf{M}(0,0)$ is $8$-dimensional and non-isomorphic to $\mathbf{M}(-1,0)$, with an explicitly described open dense subset. On the contrary to Proposition \ref{de1}, we obtain two-codimensional subspace $\mathbf{M}_{\Dd}(0,0)$ of $\mathbf{M}(0,0)$.  

\begin{proposition}
$\mathbf{M}_{\Dd}(0,0)$ contains the total space of $\Oo_{\PP^5}(-2)$ with the zero section contracted to a point, as an open dense subset. 
\end{proposition}
\begin{proof}
Take $(\Ee, \Phi)\in \mathbf{M}_{\Dd}(0,0)$. From $c_2=c_1^2$, we get that $\Ee$ is not stable and so it fits into the following exact sequence
$$0\to \Oo _{\PP^2}(t)\to \Ee \to \Ii _Z(-t)\to 0$$ 
with $t\ge 0$ and $\deg (Z)=t^2$. First assume $t>0$. Since every map $\Oo _{\PP^2}(t) \rightarrow \Ii _Z(-t)\otimes T_{\PP^2}(-\log \Dd )$ is the zero-map, we get $\Phi (\Oo _{\PP^2}(t)) \subset \Oo _{\PP^2}(t)\otimes T_{\PP^2}(-\log \Dd )$, contradicting the semistability of $(\Ee, \Phi)$. Now assume $t=0$ and so we get $\Ee \cong \Oo _{\PP^2}^{\oplus 2}$. Then we follow the argument in \cite[Theorem 5.3]{Rayan} to get the assertion. 
\end{proof}


\section{Coherent system and Holomorphic triple}\label{osem}
If $\Ff\subset \Ee$ is a non-trivial subsheaf, then its saturation $\widetilde{\Ff}$ is defined to be the maximal subsheaf of $\Ee$ containing $\Ff$ with $\rk \widetilde{\Ff}=\rk \Ff$; $\widetilde{\Ff}$ is the only subsheaf of $\Ee$ containing $\Ff$ with $\Ee /\widetilde{\Ff}$ torsion-free.

\subsection{Coherent system}
Inspired by the theory of coherent systems on smooth algebraic curves in \cite{bgmn}, we consider the following definition. Let $\Ee$ be a torsion-free sheaf of rank $r\ge 2$ on $X$ and $(\Ee, \Phi)$ be a $\Dd$-logarithmic co-Higgs structure. Then we define a set 
$$\Ss=\Ss(\Ee, \Phi):=\{ (\Ff, \Gg)~|~0\subsetneq \Ff \subseteq \Gg \subseteq \Ee \text{ with }\Phi (\Ff )\subseteq \Gg \otimes T_X(-\log \Dd )\}.$$
For a fixed real number $\alpha \ge 0$ and $(\Ff, \Gg) \in \Ss$, set
\begin{align*}
\mu _\alpha (\Ff ,\Gg) &= \mu (\Ff ) + \alpha  \left (\frac{\rk \Ff }{\rk \Gg}\right)\\
\mu '_\alpha (\Ff ,\Gg) &= \mu (\Ff ) + \alpha  \left (\frac{\rk \Ff}{\rk \Ff +\rk \Gg}\right).
\end{align*}
Note that $\mu _\alpha (\Ee ,\Ee ) =\mu (\Ee) + \alpha$ and $\mu '_\alpha (\Ee ,\Ee ) =\mu (\Ee )+ \alpha /2$. From now on we use $\mu _\alpha$, but $\mu '_\alpha$ does the same job. In general, we have $\mu _\alpha (\Ff ,\Gg )\le \mu (\Ff )+\alpha$ for $(\Ff ,\Gg )\in \Ss$ and equality holds if and only if $\rk \Ff = \rk \Gg$, i.e. $\Gg$ is contained in the saturation $\widetilde{\Ff}$ of $\Ff$ in $\Ee$. 

\begin{definition}
The pair $(\Ee ,\Phi)$ is said to be $\mu _\alpha$-stable (resp. $\mu _\alpha$-semistable) if $\mu _\alpha (\Ff,\Gg ) < \mu _\alpha (\Ee ,\Ee )$ (resp. $\mu _\alpha (\Ff,\Gg ) \le \mu _\alpha (\Ee ,\Ee )$) for all $(\Ff ,\Gg )\in \Ss \setminus \{(\Ee ,\Ee )\}$. A similar definition is given with $\mu '_\alpha$. 
\end{definition}
Note that if $\Ee$ is semistable (resp. stable), then a pair $(\Ee ,\Phi)$ is $\mu_{\alpha}$-semistable (resp. $\mu_{\alpha}$-stable) for any $\alpha$ and $\Phi$. The converse also holds for $\Phi =0$.

\begin{remark}
We have $\Phi (\Ff) \subseteq \widetilde{\Gg}\otimes T_X(-\log \Dd)$ for $(\Ff, \Gg)\in \Ss$ and so to test the $\mu_{\alpha}$-(semi)stability of $(\Ee ,\Phi)$, it is sufficient to test the pairs $(\Ff ,\Gg )\in \Ss \setminus \{(\Ee ,\Ee )\}$ with $\Gg$ saturated in $\Ee$. Moreover, if $\Gg$ is saturated in $\Ee$, then $\Gg\otimes T_X(-\log \Dd)$ is saturated in $\Ee\otimes T_X(-\log \Dd)$. Since $\Phi (\Ff)$ is a subsheaf of $\Phi (\widetilde{\Ff})$ with the same rank we have $\Phi (\widetilde{\Ff})\subseteq \Gg \otimes T_X(-\log \Dd)$. So to test the $\mu_{\alpha}$-(semi)stability of $(\Ee ,\Phi)$ it is sufficient to test the pairs $(\Ff ,\Gg )\in \Ss \setminus \{(\Ee ,\Ee )\}$ with both $\Ff$ and $\Gg$ saturated in $\Ee$.
\end{remark}

\begin{lemma}\label{lem23}
If $(\Ee ,\Phi )$ is not semistable (resp. stable), then it is not $\mu _\alpha$-semistable (resp. not $\mu _\alpha$-stable) for any $\alpha$.
\end{lemma}

\begin{proof}
Take $\Ff \subset \Ee$ such that $\Phi (\Ff )\subseteq \Ff \otimes T_X(-\log \Dd )$ and $\mu (\Ff )>\mu (\Ee )$ (resp. $\mu (\Ff )\ge \mu (\Ee )$). We have $(\Ff ,\Ff )\in \Ss$ and so $\mu _\alpha (\Ff ,\Ff )=\mu (\Ff )+\alpha >$ (resp. $\ge$) $\mu (\Ee )+\alpha=\mu_\alpha (\Ee, \Ee)$, proving the assertion. 
\end{proof}

\begin{remark}
Lemma \ref{lem23} shows that $\mu _\alpha$-stability is stronger than the stability of the pairs $(\Ee ,\Phi )$ in the sense of \cite{R1, R2, Rayan} and so they form a bounded family if we fix the Chern classes of $\Ee$. However, if $(\Ee ,\Phi )$ is not $\mu _\alpha$-semistable, a pair $(\Ff ,\Gg )\in \Ss$ with $\mu _\alpha (\Ff ,\Gg )>\mu (\Ee )+\alpha$ and maximal $\mu_\alpha$-slope may have $\rk (\Gg )>\rk (\Ff )$, i.e. $\Phi (\Ff )\nsubseteq \Ff\otimes T_X(-\log \Dd)$ and so we do not define the Harder-Narashiman filtration of $\mu_\alpha$-unstable pairs $(\Ee ,\Phi )$.
\end{remark}

\begin{proposition}\label{yyy}
Let $(\Ee, \Phi)$ be a $\Dd$-logarithmic co-Higgs bundle on $X$ with $\Ee$ not semistable. Then there exist two positive real numbers $\beta$ and $\gamma$ such that
\begin{itemize}
\item [(i)] $(\Ee, \Phi)$ is not $\mu_\alpha$-semistable for all $\alpha<\beta$, and
\item [(ii)] if $(\Ee, \Phi)$ is semistable in the sense of Definition \ref{ss1}, it is $\mu_\alpha$-semistable for all $\alpha>\gamma$. 
\end{itemize}
\end{proposition}
\begin{proof}
Assume that $\Ee$ is not semistable and take a subsheaf $\Gg$ with $\mu (\Gg )>\mu (\Ee)$. Note that $(\Gg ,\Ee)\in \Ss$. Then there exists a real number $\beta>0$ such that $\mu _\alpha (\Gg ,\Ee )> \mu (\Ee )+\alpha =\mu _\alpha (\Ee,\Ee)$ for all $\alpha$ with $0<\alpha  <\beta$. Thus $(\Ee ,\Phi)$ is not $\mu_{\alpha}$-semistable if $\alpha < \beta$. 

Now assume that $\Ee$ is not semistable, but that $(\Ee ,\Phi )$ is semistable. Define
$$\Delta=\{\text{the saturated subsheaves }\Aa \subset \Ee ~|~\mu(\Aa)>\mu(\Ee)\}.$$
Let $\mu _{\max}(\Ee)$ be the maximum of the slopes of subsheaves of $\Ee$, which exists as a finite real number by the existence of the Harder-Narasimhan filtration of $\Ee$. Since $\Ee$ is not semistable, we have $\mu _{\max}(\Ee) > \mu (\Ee)$ and set $\gamma := r(\mu _{\max}(\Ee )-\mu (\Ee))>0$. Fix any real number $\alpha \ge  \gamma$. Now take $\Aa \in \Delta$ and set $s:=\rk \Aa$. Since $(\Ee, \Phi)$ is semistable, we get $\rk \Bb >s$. Thus we have
\begin{align*}
\mu _\alpha (\Aa,\Bb) &\le \mu (\Aa) + \alpha s/(s+1) \\
&\le \mu (\Aa) +\alpha (r-1)/r \le \mu_\alpha (\Ee, \Ee) 
\end{align*}
and so $(\Ee ,\Phi)$ is $\mu_\alpha$-semistable for all $\alpha \ge \gamma$.
\end{proof}

\begin{remark}
For $s=1,\ldots ,r-1$, let $\Delta _s$ be the set of all $\Gg \in \Delta$ with rank $s$. If $\mu (\Gg )< \mu _{\max}(\Ee )$ for all $\Gg \in \Delta _{r-1}$, we may use a lower real number instead of $\gamma$ in the proof of Proposition \ref{yyy}.
\end{remark}

\begin{example}\label{++11+}
Let $X=\PP^1$ and take $\Dd=\{p\}$ with $p$ a point. Then we have $T_{\PP^1}(-\log \Dd )\cong T_{\PP^1}(-p)\cong \Oo _{\PP^1}(1)$. Let $(\Ee, \Phi)$ be a semistable $\Dd$-logarithmic co-Higgs bundle of rank $r\ge 2$ on $\PP^1$ with $\Ee \cong \oplus _{i=1}^{r} \Oo _{\PP^1}(a_i)$ with $a_1\ge \cdots \ge a_r$ and $a_i -a_{i+1} \le 1$ for all $i=1,\dots ,r+1$ as in Example \ref{bbb+1}. We assume that $\Ee$ is not semistable, i.e. $a_r<a_1$. The value $\gamma$ in Proposition \ref{yyy} could depend on $\Phi$, although it is the same for all general $\Phi$. Up to a twist we may assume $a_1=0$. We have $\mu (\Ee) = c_1/r$ with $c_1=a_1+\cdots +a_r$. For each $s=1,\ldots ,r-1$, set $b_s = (a_1+\cdots +a_s)/s$ and define
$$\gamma _0:= \max _{1\le s \le r-1} (s+1)(b_s-c_1/r).$$
We have $\mu (\Ff )\le b_s$ for all $\Ff \in \Delta _s$ and so $\mu _\alpha (\Ff, \Gg )\le \mu _\alpha (\Ee ,\Ee)$ for all $(\Ff ,\Gg)$ with $\rk \Ff  =s$ and $\Phi (\Ff )\nsubseteq \Ff\otimes T_{\PP^1}(-\log \Dd)$. Hence $(\Ee ,\Phi )$ is $\mu_\alpha$-semistable for all $\alpha \ge \gamma _0$.
\end{example}

\begin{example}
Similarly as in Example \ref{++11+}, we take $X=\PP^1$ and $\Dd =\emptyset$. Then we have $T_{\PP^1}(-\log \Dd) \cong T_{\PP^1} \cong \Oo _{\PP^1}(2)$. We argue as in Example \ref{++11+}, except that now we only require that $a_i -a_{i+1} \le 2$ for all $i=1,\ldots ,r-1$. 
\end{example}

\begin{example}
Take $X=\PP^n$ with $n\ge 2$ and assume that $(\Ee, \Phi)$ is a semistable logarithmic co-Higgs reflexive sheaf of rank two with $\Ee$ not semistable. Up to a twist we may assume $c_1(\Ee)\in \{-1,0\}$. Set $c_1:= c_1(\Ee)$. Since $\Ee$ is not semistable, we have an exact sequence
\begin{equation}
0\to \Oo _{\PP^n}(t) \to \Ee \to \Ii _Z(c_1-t)\to 0
\end{equation}
with either $Z=\emptyset$ or $\dim (Z)=n-2$, and $t\ge 0$ and $t>0$ if $c_1(\Ee)=0$. Since $(\Ee ,\Phi )$ is semistable, there is no saturated subsheaf $\Aa \subset \Ee$ of rank one with $(\Aa ,\Aa )\in \Ss$ and $\mu (\Aa )> -1$. Note that $\mu _\alpha (\Oo _{\PP^n}(t),\Ee ) =t +\alpha /2$ and so $(\Ee ,\Phi )$ is $\mu_\alpha$-stable (resp. $\mu_\alpha$-semistable) if and only if $\alpha > 2t-c_1$ (resp. $\alpha \ge 2t-c_1$). 

Now we discuss the existence of such a pair $(\Ee ,\Phi)$. Since $(\Ee ,\Phi )$ is semistable, we should have $\Phi (\Oo _{\PP^n}(t)) \nsubseteq \Oo_{\PP^n}(t) \otimes T_{\PP^n}(-\log \Dd )$
and so there is a non-zero map $\Oo _{\PP^n}(t)\rightarrow \Ii _Z(c_1-t)\otimes T_{\PP^n}(-\log \Dd )$. Since $t>c_1-t$ and $h^0(T_{\PP^n}(-2)) =0$, we get
$t=0$ and $c_1 =-1$. Then we also get $H^0(\Ii _Z(-1)\otimes T_{\PP^n}(-\log \Dd )) \ne 0$, which gives restrictions on the choice of $\Dd$ and $Z$. Assume that $\Dd=\{D\}$ with $D\in |\Oo_{\PP^n}(1)|$ a hyperplane, so that $T_{\PP^n}(-\log \Dd )\cong \Oo _{\PP^n}(1)^{\oplus n}$. In this case we get $Z=\emptyset$ and so $\Ee \cong \Oo _{\PP^n}\oplus \Oo _{\PP^n}(-1)$. See Proposition \ref{de1} for the associated moduli space in case $n=2$. 
\end{example}

\subsection{Holomorphic triple}
We may also consider a holomorphic triple of logarithmic co-Higgs bundles and define its semistability as in \cite{bgg}. 
\begin{definition}
A holomorphic triple of $\Dd$-logarithmic co-Higgs bundles is a triple $((\Ee_1,\Phi _1),(\Ee _2,\Phi _2),f)$, where each $(\Ee _i,\Phi _i)$ is a $\Dd$-logarithmic co-Higgs sheaf with each $\Ee _i$ torsion-free on $X$ and $f: \Ee _1\rightarrow \Ee _2$ is a map of sheaves such that $\Phi _2\circ f = \hat{f}\circ \Phi _1$, where $\hat{f}:\Ee _1\otimes T_X(-\log \Dd )\rightarrow \Ee _2\otimes T_X(-\log \Dd )$ is the map induced by $f$.  
\end{definition}

For any real number $\alpha \ge 0$, define the $\nu_\alpha$-slope of a triple $\Aa=((\Ee_1,\Phi) _1,(\Ee _2,\Phi _2),f)$ to be the $\nu_\alpha$-slope of the triple $(\Ee _1,\Ee _2,f)$ in the sense of \cite{bgg}, i.e.
$$ \nu _\alpha((\Ee_1,\Phi) _1,(\Ee _2,\Phi _2),f) =\frac{\deg_{\alpha}(\Aa)}{\rk \Ee_1 + \rk \Ee_2},$$
where $\deg_{\alpha}(\Aa)= \deg (\Ee_1)+\deg (\Ee_2)+\alpha \rk \Ee_1$. A holomorphic subtriple $\Bb = ((\Ff_1,\Psi _1),(\Ff _2,\Psi _2),g)$ of $\Aa = ((\Ee_1,\Phi _1),(\Ee _2,\Phi _2),f)$ is a holomorphic triple with $\Ff _i\subset \Ee _i$, $\Psi _i = \Phi _{i|\Ff _i}$ and $g = f_{|\Ff _1}$. Since $\Phi _i$ is integrable, so is $\Psi _i$.

\begin{remark}
As before, we may use the slope $\nu _\alpha$ to define the $\nu_\alpha$-(semi)stability for $\Dd$-logarithmic co-Higgs triples. If $h: \Aa \rightarrow \Bb$ is a non-zero map of $\nu_\alpha$-semistable holomorphic triples, then we have $\nu _\alpha (\Bb )\ge \mu _\alpha (\Aa)$. Moreover, if $\Aa$ is $\nu_\alpha$-stable, then either $\nu _\alpha (\Bb )>\nu _\alpha (\Aa)$ or $h$ is injective; in addition, if $\Bb$ is also $\nu_\alpha$-stable, then $h$ is an automorphism.
\end{remark}

\begin{remark}\label{rem25}
The degenerate holomorphic triple $((\Ee_1, \Phi_1), (\Ee_2, \Phi_2), 0)$ with $f=0$ is $\nu_\alpha$-semistable if and only if $\alpha=\mu (\Ee_2)-\mu(\Ee_1)$ and both $(\Ee_i, \Phi_i)$'s are semistable as in \cite[Lemma 3.5]{bg}. Moreover such triples are not $\nu_\alpha$-stable (see \cite[Corollary 3.6]{bg}). Note that if $\Phi_1=\Phi_2=0$, then we fall into the usual holomorphic triples. We also have an analogous statement for the case $r_2=\rk \Ee_2=1$ as in \cite[Lemma 3.7]{bg}. 
\end{remark}

\begin{remark}
For subtriples $\Bb$ and $\Bb'$ of $\Aa$, we may define their sum and intersection $\Bb+\Bb'$ and $\Bb \cap \Bb'$; let $\Bb =  ((\Ff_1,\Psi _1),(\Ff _2,\Psi _2),g)$ and $\Bb '=  ((\Ff '_1,\Psi '_1),(\Ff '_2,\Psi '_2),g)$. Then we may use $\Ff _i+\Ff '_i$ and $\Ff _i\cap \Ff '_i$ with the restrictions of $\Phi _i$ and $f$ to them. Now call $\widetilde{\Ff}_i$ the saturation of $\Ff _i$ in $\Ee _i$. Since $\widetilde{\Ff}_i\otimes T_X(-\log \Dd )$ is saturated in $\Ee _i\otimes T_X(-\log \Dd)$, we have $\Phi _i(\widetilde{\Ff}_i)\subseteq \widetilde{\Ff}_i\otimes T_X(-\log \Dd )$. Since $f(\Ff _1)\subseteq
\widetilde{\Ff}_2$, we have $f(\widetilde{\Ff}_1)\subseteq \widetilde{\Ff}_2$ and so we may also define the saturation $\widetilde{\Bb}$ of $\Bb$ with $\nu_{\alpha}(\widetilde{\Bb}) \ge \nu_{\alpha}(\Bb)$.
\end{remark}

Fix $\alpha \in \RR_{>0}$ and let $\Aa = ((\Ee_1,\Phi _1),(\Ee _2,\Phi _2),f)$ be a holomorphic triple. We define $\beta(\Aa)$ to be the maximum of the set of the $\nu_\alpha$-slopes of all subtriples of $\Aa$ and let
$$\BB:=\{\Bb \subseteq \Aa~|~\nu_\alpha(\Bb)=\beta(\Aa)\}.$$

\begin{lemma}\label{kk1}
The set of the $\nu_\alpha$-slopes of all subtriples of $\Aa$ is upper bounded and so $\beta(\Aa)$ exists. Moreover, the set $\BB$ has a unique maximal element 
\end{lemma}

\begin{proof}
The ranks of any non-zero subsheaf of $\Ee _i$ is upper bounded by $r_i:=\rk \Ee_i$ and lower bounded by $1$. The existence of the Harder-Narasimhan filtration of $\Ee _i$ gives the existence of positive rational numbers $\gamma_i$ with denominators between $1$ and $r_i$ such that $\mu (\Ff _i)\le \gamma _i$ for all non-zero subsheaves $\Ff _i$ of $\Ee_i$. We may use the definition of $\nu_\alpha$-slope to get an upper-bound for the $\nu_\alpha$-slopes of the subtriples of $\Aa$. There are only finitely many possible $\nu_\alpha$-slopes greater than $\nu _\alpha (\Aa)$, because the ranks are upper and lower bounded and each $\deg (\Gg)$ for a subsheaf $\Gg$ of $\Ee _i$ is an integer, upper bounded by $\max \{r_1\mu (\Ee _1), r_2\mu (\Ee _2)\}$. Thus the set of the $\nu_\alpha$-slopes of all subtriples of $\Aa$ has a maximum $\beta(\Aa)$. 

If $\nu _\alpha (\Aa )=\beta(\Aa)$, then $\Aa$ itself is the maximum element of $\BB$. Now assume $\nu _\alpha (\Aa )>\delta$ and that there are $\Bb _1,\Bb _2\in \BB$ with each $\Bb _i$ maximal and $\Bb _1\ne \Bb _2$. Since $\Bb _i$ is maximal, it is saturated and so $\Aa _i:= \Aa /\Bb _i$ is a holomorphic triple for each $i$. Since $\Bb _2\ne \Bb _1$, the inclusion $\Bb _2\subset \Aa$ induces a non-zero map $u: \Bb_2\rightarrow \Aa /\Bb _1$. Since $\nu _\alpha (\mathrm{ker}(u)) \le \beta(\Aa)$ if $u$ is not injective, we have $\nu _\alpha (u(\Aa /\Bb _1)) \ge \beta(\Aa)$. Thus we get $\nu _\alpha (\Bb _1+\Bb _2) \ge \beta(\Aa)$, contradicting the maximality of $\Bb _1$ and the assumption $\Bb _2\ne \Bb _1$.
\end{proof}

Assume that $\Aa$ is not $\nu_\alpha$-semistable. By Lemma \ref{kk1} there is a subtriple $D(\Aa) =  ((\Ff_1,\Psi _1),(\Ff _2,\Psi _2),g)\in \BB$ such that every $\Gg \in \BB$ is a subtriple of $D(\Aa)$ and each $\Ff _i$ is saturated in $\Ee_i$. Note that $D(\Aa)$ is $\nu_\alpha$-semistable. Since $\Ff _i$ is saturated in $\Ee _i$ and $\Psi _i = \Phi _{i|\Ff _i}$ for each $i$, $\Phi _i$ induces a co-Higgs field $\tau _i: \Ee _i/\Ff _i\rightarrow (\Ee _i/\Ff _i)\otimes T_X(-\log \Dd )$. Since $\Phi _i$ is integrable, so is $\tau _i$. Since $g=f_{|\Ff _1}$, $f$ induces a map $f': \Ee _1/\Ff _1\rightarrow \Ee _2/\Ff _2$ such that $\Aa /D(\Aa) := ((\Ee _1/\Ff _1, \tau _1),(\Ee _2/\Ff _2,\tau _2),f')$ is a holomorphic triple. Now we may check that each subtriple of $\Aa /D(\Aa)$ has $\nu_\alpha$-slope less than $\beta(\Aa)$ and so $D(\Aa)$ defines the first step of the Harder-Narasimhan filtration of $\Aa$.The iteration of this process allows us to have the Harder-Narasimhan filtration of $\Aa$ with respect to $\nu _\alpha$.

\begin{corollary}\label{mcor}
Any holomorphic triple admits the Harder-Narasimhan filtration with respect to $\nu_\alpha$-slope. 
\end{corollary}

\begin{remark}\label{rem66}
Let $Z$ denote a projective completion of $T_X(-\log \Dd)$, e.g. $Z=\PP (\Oo_X\oplus T_X(-\log \Dd))$, and call $D_{\infty}:=Z\setminus T_X(-\log \Dd)$ the divisor at infinity. By \cite[Lemma 6.8]{s} a co-Higgs sheaf $(\Ee, \Phi)$ on $X$ is the same thing as a coherent sheaf $\Ee_Z$ with $\mathrm{Supp}(\Ee_Z) \cap D_{\infty}=\emptyset$. Due to \cite[Corollary 6.9]{s} we may interpret a $\nu_\alpha$-semistable holomorphic triple of logarithmic co-Higgs bundles on $X$ as a $\nu_\alpha$-semistable holomorphic triple of vector bundles on $Z$ with support not intersecting $D_{\infty}$ as in \cite{bgg}. 
\end{remark}

Based on Remark \ref{rem66}, we may consider a $\nu_\alpha$-semistable triple of $\Dd$-logarithmic co-Higgs sheaves as a $\nu_\alpha$-semistable quiver sheaf for the quiver $\xymatrix{\stackrel{1}{\circ} \ar[r]& \stackrel{2}{\circ}}$ on $Z$ with empty intersection with $D_{\infty}$. This interpretation ensures the existence of  moduli space of $\nu_\alpha$-stable triples of $\Dd$-logarithmic co-Higgs sheaves on $X$, say $\mathbf{M}_{\Dd, \alpha}(r_1, r_2, d_1, d_2)$ with $(r_i, d_i)$ a pair of rank and degree of the $i^{\mathrm{th}}$-factor of the triples; indeed we may consider Gieseker-type semistability of quiver sheaves to produce the moduli space as in \cite{Schmitt}. As noticed in \cite[Remark in page 17]{Schmitt}, the $\nu_\alpha$-stability implies the Gieseker-type stability and so $\mathbf{M}_{\Dd, \alpha}(r_1, r_2, d_1, d_2)$ can be considered as a quasi-projective subvariety of the one in \cite{Schmitt}. Now let us define 
$$\alpha_m:=\mu(\Ee_2)-\mu(\Ee_1)~,~\alpha_M:=\left( 1+ \frac{r_1+r_2}{|r_1-r_2|}\right) \left(\mu(\Ee_2)-\mu(\Ee_1)\right)$$ 
for $\Aa = ((\Ee_1,\Phi _1),(\Ee _2,\Phi _2),f)$ as in \cite{PP}. Then we have

\begin{proposition}\cite[Proposition 2.2]{bgg}
If $\alpha>\alpha_M$ with $\rk \Ee_1\ne \rk \Ee_2$ or $\alpha<\alpha_m$, then there exists no $\nu_\alpha$-semistable triple of $\Dd$-logarithmic co-Higgs sheaves.
\end{proposition}
\begin{proof}
Due to \cite[Corollary 6.9]{s}, it is sufficient to check the assertion for $\nu_\alpha$-semistability for a triple of coherent sheaves on $Z$. While the proof of \cite[Proposition 2.2]{bgg} is for curves, the proof is numerical involving rank and degree with respect to a fixed ample line bundle so that it works also for $Z$. 
\end{proof}

From now on we assume that $X$ is a smooth projective curve of genus $g$ and let $\Dd=\{p_1, \ldots, p_m\}$ be a set of $m$ distinct points on $X$. Take $g\in \{0,1\}$ and assume that $T_X(-\log \Dd) \cong \Oo_X$, i.e. $(g,m)\in \{(0,2), (1,0)\}$. For any triple $\Aa= ((\Ee _1,\Phi _1),(\Ee _2,\Phi _2),f)$ and $c\in \CC$, set 
\begin{equation}\label{fam}
\Aa _c:= ((\Ee _1,\Phi _1 -c {\cdot}\mathrm{Id}_{\Ee_1}),(\Ee _2,\Phi _2-c{\cdot}\mathrm{Id}_{\Ee_2}),f)
\end{equation}
and then $\Aa_c$ is also a triple. In particular, if $\Ee _1\cong \Ee _2$ and $f \cong c {\cdot}\mathrm{Id}_{\Ee_1}$, then the study of the $\nu _\alpha$-(semi)stability of $\Aa$ is reduced to the known case $f=0$. 

\begin{remark}\label{uu1}
Assume that $f$ is not injective. Since $\hat{f}\circ \Phi _1 =\Phi _2\circ f$, we have $\Phi _1(\mathrm{ker}(f)) \subseteq \mathrm{ker}(\hat{f})$ and $\Bb:= ((\mathrm{ker}(f),\Phi _{1|\mathrm{ker}(f)}),(0,0),0)$ is a subtriple of $\Aa$. Set $\rho := \rk (\mathrm{ker}(f))$ and $\delta:= \deg (\mathrm{ker}(f))$. If we have 
$$\nu _\alpha (\Bb ) = \delta /\rho + \alpha> \frac{r_1\alpha +d_1+d_2}{r_1+r_2},$$
then $\Aa$ would not be $\nu_\alpha$-semistable.
\end{remark}

\begin{remark}\label{xxx1}
For any triple $\Aa =((\Ee _1,\Phi _1),(\Ee _2,\Phi _2),f)$, we get a dual triple $\Aa ^\vee = ((\Ee _2^\vee ,\Phi _2^\vee), (\Ee _1^\vee ,\Phi _1^\vee ),f^\vee)$, where $\Phi_i^\vee$ and $f^\vee$ are the transpose of $\Phi_i$ and $f$, respectively. Then $\Aa $ is $\nu _\alpha$-(semi)stable if and only if $\Aa ^\vee$ is $\nu _\alpha$-(semi)stable (see \cite[Proposition 3.16]{bg}).
\end{remark}

\begin{remark}\label{xxx2}
Assume $(g,m)=(1,0)$ and take a triple $\Aa=((\Ee_1, \Phi_1), (\Ee_2, \Phi_2), f)$ with each $\Ee _i$ simple. By Atiyah's classification of vector bundles on elliptic curves, the simpleness of $\Ee_i$ is equivalent to its stability and also equivalent to its indecomposability and with degree and rank coprime. Then each $\Phi _i$ is the multiplication by a constant, say $c_i$. We get that the two triples $\Aa$ and $((\Ee_1, 0), (\Ee_2, 0), f)$ share the same subtriples and so these two triples are $\nu _\alpha$-(semi)stable for the same $\alpha$ simultaneously. There is a good description of this case in \cite[Section 7]{PP0}.
\end{remark}

Now we suggest some general description on $\nu_\alpha$-(semi)stable triples on $X$ in case of $r_1=r_2=2$ from $(a)\sim (c)$ below; we exclude the case described in Remark
\ref{xxx2} and silently use Remark \ref{xxx1} to get a shorter list.  In some case we stop after reducing to a case with $f$ not injective, i.e. to a case in which $\Aa$ is not $\nu _\alpha$-semistable for $\alpha \gg 0$ (see Remark \ref{uu1}).

\quad (a) Assume $r_1=r_2=2$ and that at least one of $\Ee _i$ is not semistable, say $\Ee_1$. Then, due to Segre-Grothendieck theorem and Atiyah's classification of vector bundles on elliptic curves, we have $\Ee _1 \cong \Ll _1\oplus \Rr _1$ with $\deg (\Ll _1)>\deg (\Rr _1)$ and $\Ee _2 \cong \Ll _2\oplus \Rr _2$ with $\deg (\Ll _2)\ge \deg (\Rr _2)$, or $g=1$ and $\Ee _2$ is a non-zero extension of the line bundle $\Ll _2$ by itself; in the latter case we put $\Rr _2:= \Ll _2$. If $\Ee _2$ is indecomposable, then it has a unique line bundle isomorphic to $\Ll _2$ and so $\Phi _2(\Ll _2)\subseteq \Ll _2$. We have  
$$\nu _\alpha (\Aa ) =\alpha /2 + (\deg (\Ll _1)+\deg (\Ll _2)+\deg (\Rr _1)+\deg (\Rr _2))/4.$$
The map $\Phi _i: \Ee _i\rightarrow \Ee _i$ induces a map $\Phi _{i|\Ll _i}: \Ll _i\rightarrow \Ll _i$, which is induced by the multiplication by a constant, say $c_i$. Then we get two triples $\Aa _{c_i}$ for $i=1,2$. Since $\Aa _{c_2}$ is a triple, we get $f(\Ll _1)\subseteq \Ll _2$ and so we may define a subtriple $\Aa_1:= ((\Ll _1,\Phi _{1|\Ll _1}),(\Ll _2,\Phi _{2|\Ll _2}),f_{|\Ll _1})$ with
\begin{align*}
\nu _\alpha (\Aa _1)& = \alpha +(\deg (\Ll_1) +\deg (\Ll _2))/2 \\
&> \alpha /2 + (\deg (\Ll _1)+\deg (\Ll _2)+\deg (\Rr _1)+\deg (\Rr _2))/4=\nu_\alpha (\Aa),
\end{align*}
which implies that $\Aa$ is not $\nu _\alpha$-semistable.

 \quad (b) Form now on we assume that $\Ee _1$ and $\Ee _2$ are semistable. We also assume that $f$ is non-zero so that $\mu (\Ee _1)\le \mu (\Ee _2)$. We are in a case with $r_1=r_2=2$ and we look at a proper subtriple $\Bb = ((\Ff _1,\Phi _{1| \Ff _1}), (\Ff _2,\Phi _{2| \Ff _2}),f_{|\Ff _1})$ with maximal $\nu _\alpha (\Bb )$. In particular, each $\Ff _i$ is saturated in $\Ee _i$, i.e. either $\Ff _i =\Ee _i$ or $\Ff _i=0$ or $\Ee _i/\Ff _i$ is a line bundle. Set $s_i:= \rk (\Ff _i)$ and then we have $1\le s_1+s_2\le 3$. If $s_2=2$, i.e. $\Ff _2 = \Ee _2$, then we have $\nu _\alpha (\Bb )< \nu _\alpha (\Aa)$ for all $\alpha >0$, because $\Ee _1$ is semistable and $\mu (\Ee _1) \le \mu (\Ee _2)$. If $s_2=0$, then $f$ is not injective. If $s_1=0$ we just exclude the case $\alpha \le \alpha _m$ with subtriple $((0,0),(\Ee _2,\Phi _2),0)$. In the case $s_1=s_2=1$ we know that $\nu _\alpha (\Bb )\le \nu _\alpha (\Aa)$ and that equality holds if and only if both $\Ee _1$ and $\Ee _2$ are strictly semistable and each $\Ff _i$ is a line subbundle of $\Ee _i$ with maximal degree. Note that the injectivity of $f$ implies $s_1\le s_2$. Thus when $f$ is injective, it is sufficient to test the case $s_1=s_2=1$. Then we have the following, when $f$ is injective. 
 \begin{itemize}
\item If $\alpha >\a_m$ and at least one of $\Ee_i$'s is stable, then $\Aa$ is $\nu_\alpha$-stable 
\item If $\alpha \ge \a_m$ and $\Ee _1$ and $\Ee _2$ are semistable, then $\Aa$ is $\nu _\alpha$-semistable. 
\item If $\alpha > \a_m$ and $\Ee _1$ and $\Ee _2$ are strictly semistable, then $\Aa$ is strictly $\nu _\alpha$-semistable if and only if there are maximal degree line bundles $\Ll _i\subset \Ee _i$ such that $\Phi _i(\Ll _i)\subseteq \Ll _i$ for each $i$ and $f(\Ll _1)\subseteq \Ll _2$.
\end{itemize}

\begin{lemma}\label{uu22}
For a general map $f: \Ee_1 \rightarrow \Ee_2$ with $\Ee_i:=\Oo_{\PP^1}(a_i)^{\oplus 2}$ and $a_2\ge a_1+2$, there exists no subsheaf $\Oo _{\PP^1}(a_1) \subset \Ee _1$ such that the saturation of its image in $\Ee _2$ is a line bundle isomorphic to $\Oo _{\PP^1}(a_2)$.
\end{lemma}

\begin{proof}
Up to a twist we may assume that $a_1=0$. If we fix homogeneous coordinates $x_0,x_1$ on $\PP^1$, then the map $f$ is induced by two forms $u(x_0,x_1)$ and $v(x_0,x_1)$ of degree $a_2$. Then it is sufficient to prove that there is no point $(a,b)\in \CC^2\setminus \{(0,0)\}$ with which $au(x_0,x_1)+bv(x_0,x_1)$ is either identically zero or with a zero of multiplicity $a_2$. This is true for general $u(x_0,x_1)$ and $v(x_0,x_1)$, e.g. we may take $u(x_0,x_1) = x_0^{a_2}+x_0x_1^{a_2-1}$ and $v(x_0,x_1) = x_0x_1^{a_2-1} +x_1^{a_2}$.
\end{proof}

The next is an analogue of Lemma \ref{uu22} for elliptic curves. 

\begin{lemma}\label{uu23}
Let $X$ be an elliptic curve with two line bundles $\Ll_i$ for $i=1,2$ such that $\deg (\Ll_2)\ge \deg (\Ll_1)+4$. For a general map $f: \Ll_1^{\oplus 2} \rightarrow \Ll_2^{\oplus 2}$, there is no subsheaf $\Ll_1 \subset \Ll_1^{\oplus 2}$ such that the saturation of its image in $\Ll_2^{\oplus 2}$ is isomorphic to $\Ll_2$. 
\end{lemma}

\begin{proof}
It is sufficient to find an injective map $h: \Ll_1^{\oplus 2} \rightarrow \Ll_2^{\oplus 2}$ for which no subsheaf $\Ll_1\subset \Ll_1^{\oplus 2}$ has its image under $h$ whose saturation in $\Ll_2^{\oplus 2}$ is isomorphic to $\Ll_2$. Up to a twist we may assume $\Ll _1\cong \Oo _X$ and so $l:=\deg (\Ll_2)\ge 4$. First assume $l=4$ and write $\Ll_2 \cong \Mm ^{\otimes 2}$ with $\deg (\Mm )=2$. If $\phi : X\rightarrow \PP^1$ be a morphism of degree two, induced by $|\Mm |$, then we may set $h:=\phi^\ast (h_1)$ for a general $h_1: \Oo _{\PP^1}^{\oplus 2} \rightarrow \Oo _{\PP^1}(2)^{\oplus 2}$ with Lemma \ref{uu22} applied to $h_1$. 

Now assume $l\ge 5$ and fix an effective divisor $D\subset X$ of degree $l-4$. Then we may take as $h$ the composition of a general map $\Oo_X^{\oplus 2} \rightarrow \Ll_2(-D)^{\oplus 2}$ with the map $\Ll_2(-D)^{\oplus 2} \rightarrow \Ll_2^{\oplus 2}$ obtained by twisting with $\Oo_X(D)$. 
\end{proof}

\begin{remark}\label{tre}
Let $\Dd$ be an arrangement with $T_X(-\log \Dd )\cong \Oo _X$ on $X$ with arbitrary dimension. For two line bundles $\Ll_1$ and $\Ll_2$ with $\Ll_2\otimes \Ll_1^\vee$ globally generated, set a triple $\Bb = ((\Ee _1,0),(\Ee _2,0),f)$ with $\Ee _i\cong \Ll_i ^{\oplus r}$ and $f$ injective. As in (\ref{fam}) we may generate other triples $\Bb _c$ for each $c\in \CC$, but often there are no other $\Dd$-logarithmic co-Higgs triples with $\Bb$ as the associated triple of vector bundles. For example, assume $X$ is a smooth projective curve of genus $g\in \{0,1\}$. For a fixed co-Higgs field $\Phi _1: \Ee _1\rightarrow \Ee _1$ with the associated $(r\times r)$-matrix $A_1$ of constants, we are looking for $f$ and $\Phi _2: \Ee _2\rightarrow \Ee _2$ with the associated matrix $A_2$ such that $\Aa = ((\Ee _1,\Phi _1),(\Ee _2,\Phi _2),f)$ is a $\Dd$-logarithmic co-Higgs triple. Let $M$ be the $(r\times r)$-matrix with coefficient in $H^0(\Ll_2 \otimes \Ll_1^\vee )$ associated to $f$. Then we need $A_2$ and $M$ such that $A_2M = MA_1$. Assume that $A_1$ has a unique Jordan block. If $\Ll_1 \cong \Ll_2$ and $M$ is general, then we get a $\Dd$-logarithmic co-Higgs triple if and only if $A_2$ is a polynomial in $A_1$. If $\Ll_1 \not\cong \Ll_2$ and $f$ is general, then there is no such $A_2$. We check this for the case $r=2$ and the general case can be shown similarly. With no loss of generality we may assume that the unique eigenvalue of $A_1$ is zero. Assume the existence of $f$ and $\Phi_2$ with associated $M$ and $A_2$. We have $\ker (\Phi_1)\cong \Ll_1$ and $f(\ker (\Phi_1))\subseteq \ker (\Phi_2)$. Thus we get that $f(\Ll_1)$ has $\ker (\Phi_2) \cong \Ll_2$ as its saturation, contradicting Lemmas \ref{uu22} and \ref{uu23} for a general $f$.
\end{remark}

\begin{remark}
In the same way as in \cite{ACG} one can define $\Dd$-logarithmic co-Higgs holomorphic chains with parameters, but if the maps are general, then very few logarithmic co-Higgs fields $\Phi _i$ are allowed.
\end{remark}


\providecommand{\bysame}{\leavevmode\hbox to3em{\hrulefill}\thinspace}
\providecommand{\MR}{\relax\ifhmode\unskip\space\fi MR }
\providecommand{\MRhref}[2]{%
  \href{http://www.ams.org/mathscinet-getitem?mr=#1}{#2}
}
\providecommand{\href}[2]{#2}

\end{document}